\documentclass[preprint]{elsarticle}

\usepackage[applemac]{inputenc}
\usepackage{amsmath,amsthm,amssymb}

\usepackage{latexsym}
\usepackage[dvipsnames]{xcolor}
\usepackage{tikz}
\usetikzlibrary{arrows}
\usetikzlibrary{patterns}
\usetikzlibrary{shapes}
\usepackage{mathrsfs}
\usepackage{mathtools}
\usepackage{verbatim}
\usepackage{bbm}
\usepackage{hyperref}
\hypersetup{
	colorlinks,
	linkcolor={red!50!black},
	citecolor={blue!50!black},
	urlcolor={blue!80!black}
}
\usepackage[mathcal]{euscript}

\newcommand{\R}{\mathbb{R}}
\newcommand{\N}{\mathbb{N}}

\newcommand{\INT}{\mathrm{int}}

\newcommand{\Ll}{\mathscr L}
\newcommand{\Ha}{\mathscr H}

\newcommand{\C}{\complement}

\usepackage[usenames,dvipsnames]{pstricks}
\usepackage{epsfig,xcolor}
\usepackage{pst-grad} 
\usepackage{pst-plot} 

\newtheorem{thm}{Theorem}
\newtheorem{lemma}{Lemma}
\newtheorem{cor}[lemma]{Corollary}
\newtheorem{prop}[lemma]{Proposition}

\theoremstyle{definition}
\newtheorem{example}[lemma]{Example}

\newtheorem{defi}[lemma]{Definition}
\newtheorem{remark}[lemma]{Remark}

\newcommand{\be}{\begin{equation}}
\newcommand{\ee}{\end{equation}}

\usepackage{mathtools} 
\mathtoolsset{showonlyrefs,showmanualtags} 

\numberwithin{equation}{section}

\renewcommand{\epsilon}{\varepsilon}

\newcommand{\commentmargin}[1]{\marginpar{\tiny\textit{#1}}}
\newcommand{\Dnote}[1]{\commentmargin{\color{red} D: #1}}
\newcommand{\Lnote}[1]{\commentmargin{\color{blue} L: #1}}
\newcommand{\FR}[1]{\commentmargin{\color{ForestGreen} FR: #1}}

\begin{document}
\title{On the Lebesgue measure of the boundary of the evoluted set\tnoteref{f1}}
\tnotetext[f1]{LC, FR and DV are partially supported by the Gruppo Nazionale per l'Analisi Matematica, la Probabilit\`a\ e le loro Applicazioni (GNAMPA) of the Istituto Nazionale di Alta Matematica (INdAM). FR is partially supported by the University of Padua under the STARS Grants programme CONNECT: Control of Nonlocal Equations for Crowds and Traffic models.}

\author[1]{Francesco Boarotto}

\ead{francesco.boarotto@math.unipd.it}

\author[1]{Laura Caravenna}
\ead{laura.caravenna@unipd.it}

\author[1]{Francesco Rossi\corref{cor1}}
\ead{francesco.rossi@math.unipd.it}

\author[1]{Davide Vittone}
\ead{davide.vittone@unipd.it}

\address[1]{Dipartimento di Matematica ``Tullio Levi-Civita'', Universit\`a\ degli Studi di Padova, Via Trieste 63, 35121 Padova, Italy}
	
	\begin{abstract} The evoluted set is the set of configurations reached from an initial set via a fixed flow for all times in a fixed interval. We find conditions on the initial set and on the flow ensuring that the evoluted set has negligible boundary (i.e. its Lebesgue measure is zero). We also provide several counterexample showing that the hypotheses of our theorem are close to sharp.
	\end{abstract}
	

	\cortext[cor1]{Corresponding author}
	
	\maketitle
	
	\section{Introduction}

	The study of the attainable set from a point is a crucial problem in control theory, starting from the classical orbit, Rashevsky-Chow and Krener theorems, see \cite{agra,jurdj,krener}. If the initial state is not precisely identified, but lies in a given set, the problem gets even more complicated. The goal of this article is to study such problem in a first, simplified setting.
	
	Here, we consider a fixed Lipschitz vector field $v(x)$ acting on sets via the flow $\Phi^v_t$ it generates. 
Given an initial set $S$, we aim to describe the evoluted set 
\begin{equation}\label{intro-ev}
S^t:=\cup_{\tau\in[0,t]}\Phi^v_\tau(S),
\end{equation} that is the set of points reached at times $\tau\in[0,t]$. It was studied e.g. in \cite{colombo,lorenz,pogo}, with the name ``funnel'' too. One can also observe that $S^t$ is the attainable set at time $t$ for the control problem
\begin{eqnarray}
&&\dot x(t)=u(t)v(x(t)),\qquad\qquad u(t)\in[0,1],\qquad\qquad x(0)\in S.
\label{e-control}
\end{eqnarray}

A first question needs to be answered:

{\it If the initial set $S$ has a negligible boundary (i.e. its Lebesgue measure is zero), does the evoluted set have a negligible boundary too?}
	
	The first, striking result of this article is to provide a negative answer to this question. In Example~\ref{ex-plane} below, we exhibit a set $S\subset\R^2$ with negligible boundary and such that the boundary of the evoluted set $S^t$ is not negligible. Even more surprisingly, the counterexample relies on very low regularity of $S$, while the vector field $v$ is constant (so extremely regular). Another counterexample was provided in \cite{BoarottoRossi}.
	
	We then turn our attention to find regularity properties of $S$ that ensure that the evoluted set keeps having negligible boundary. Our main result is based on the definition of Lebesgue point for the complement $S^\C$ of a set $S$, that we recall here.
	
	\begin{defi}[Lebesgue point of a set] \label{def_Lebpt}
Let $S\subset\R^n$ be a set. We say that $x\in \R^n$ is a Lebesgue point of the complement $S^\complement$ if 
\be 
\lim_{r\to 0^+}\frac{\Ll^n( B(x,r)\setminus S)}{\Ll^n(B(x,r))}=1.
\ee
\end{defi}

Our main result is given below. As customary, $\Ha^k$ denotes the $k$-dimensional Hausdorff measure in $\R^n$.
\begin{thm}\label{teo_principale}
If $v$ is a Lipschitz vector field in $\R^n$ and $S\subset\R^n$ is any set such that $\Ll^n(\partial S)=0$ and $\Ha^{n-1}$-almost every point of its boundary $\partial S$ is not a Lebesgue point of $S^\complement$, then $\Ll^n(\partial(S^t))=0$ for every $t\in\R$.
\end{thm}

Several corollaries and the complete proof are given in Section~\ref{s-main} below. In particular, regularity hypotheses of Theorem~\ref{teo_principale} hold for several classes of interest for $S$, such as Lipschitz or smooth domains. Moreover, the same result holds for the evoluted set where initial and final times are excluded, i.e. replacing $\tau\in(0,t)$ in~\eqref{intro-ev}. 

\begin{remark}
The condition $\Ll^n(\partial S)=0$ in Theorem~\ref{teo_principale} seems natural, observing that, when $v\equiv 0$, one has $S^t=S$ and $\Ll^n(\partial (S^t))=\Ll^n(\partial S)$.
Even with non-vanishing vector fields, the thesis would not of course hold at $t=0$.
\end{remark}

\begin{remark} The assumptions in Theorem~\ref{teo_principale} look satisfactory, as they are satisfied even by open sets with wild (fractal) boundaries. In fact, the boundary of an open set $S$ might have fractal Hausdorff dimension even if no point of such boundary is a Lebesgue point of $S^\complement$, and thus necessarily $\Ll^n(\partial S)=0$. A well-known example on the plane is the open region $S\subset\R^2$ enclosed by the classical von Koch (closed snowflake) curve (see e.g.~\cite{Falconer}), whose Hausdorff dimension is $\log_3 4$.

%
One can use similar fractal constructions to produce bounded open sets $S\subset\R^n$ such that the Hausdorff dimension of $\partial S$ is any desired number in $(n-1,n)$, but such that no point of $\partial S$ is a Lebesgue point of $S^\complement$. Eventually, a countable union of such sets produces open sets $S\subset\R^n$ such that $\partial S$ has Hausdorff dimension $n$, but such that no point of $\partial S$ is a Lebesgue point of $S^\complement$.
\end{remark}

As for infinite-time horizon control problems, one might ask if Theorem~\ref{teo_principale} hold for $t=+\infty$ too, i.e. by considering the union over $\tau\in[0,+\infty)$. The answer is, surprisingly once again, negative. We provide counterexamples in Examples~\ref{ex-unbounded}--\ref{ex-infty}.

The question about evoluted sets is even more interesting when it is interpreted in terms of densities solving a Partial Differential Equation, eventually with internal control. Assume to have an initial state that is a probability measure on $S$, e.g. because the initial configuration is not precisely identified, and consider $S$ to be some form of ``safety region''. If the measure is absolutely continuous with respect to the Lebesgue measure, is it true that time modulations of the flow~\eqref{e-control} do not concentrate mass along the boundary, then eventually giving a non-zero probability of being close to unsafe configurations? This problem has been addressed in \cite{controllabilityTPDE,mintime}. This is also one of the main motivations and future applications of the result presented here: to develop a theory describing the reachable set starting from a measure under control action. See further results in this direction in \cite{bonnet2,bonnet1}.

{The results presented in this article can be seen as the final and most general ones, among a list of contributions by the authors. Indeed:
\begin{itemize}
\item in \cite{mintime}, we proved that $\Ll^n(\partial(S^t))=0$ under the stronger condition of $S$ being open and satisfying the so-called uniform interior cone condition. The proof relies on a Gronwall argument, adapted to evolution of interior cones under flow action.
\item In \cite{BoarottoRossi}, we prove the same property, but under the weaker hypothesis of $S$ being a $C^{1,1}$ domain. The proof is based on a careful study of the evolution of the normal to the boundary under flow action.
\item In the present paper, we understand that the same result holds even just assuming that $\Ha^{n-1}$-almost each point of the boundary is not a Lebesgue point of the complementary set, see Theorem \ref{teo_principale}. The proof couples measure theory with properties of evolution of non-Lebesgue points under flow action, recalling that the flow is bi-Lipschitz.\\
\end{itemize}}

	The structure of the article is the following: in Section~\ref{s-prel} we collect the main definitions, as well as known results. Section~\ref{s-ex} collects counterexamples to questions introduced above, eventually showing sharpness of the hypotheses in our main theorem. Section~\ref{s-main} presents the proof of our main result, that is Theorem~\ref{teo_principale}.

	\section{Definitions and preliminary results} \label{s-prel}

In this article, we will use the following notation: $S$ is a general subset of $\R^n$, with $S^\C:=\R^n\setminus S$ being its complement. We also define its topological closure $\bar S$, its interior $\INT(S)$ and its boundary $\partial S:=\bar S\setminus \INT(S)$.

We will deal with Lipschitz functions and vector fields. More precisely, we will deal with \emph{globally Lipschitz functions}, i.e. functions for which there exists $L>0$ such that the Lipschitz condition $|f(x)-f(y)|\leq L |x-y|$ holds for all $x,y\in \R^n$. Most of the results can be translated to the local setting, provided that more conditions ensuring some form of compactness are added. See e.g. Corollary~\ref{c-c} below. 

We will also use the following definition of Lipeomorphism, also known as bi-Lipschitz function.
	
	\begin{defi}\label{defi:lipeo}
		Given $U,V\subset \R^n$, a map $F:U\to V$ is said to be a {\em Lipeomorphism} if $F$ is a Lipschitz continuous bijective map from $U$ to $V$ and its inverse $F^{-1}$ is Lipschitz as well.
	\end{defi}

We now recall the definition of the flow $\Phi^v_t$ of a vector field $v$.

\begin{defi}[Flow of a vector field] 
Let $t\in\R$ and let $v$ be a Lipschitz vector field on $\R^n$. The flow $\Phi_t^v:\R^n\to \R^n$ along $v$ at time $t$ is the map $\Phi_t^v(x_0):=x(t)$ that returns the value at time $t$ of the (unique) solution to the Cauchy problem
	\be
		\left\{
			\begin{aligned}
					\dot{x}(t)&=v(x(t)),\\
					x(0)&=x_0.
			\end{aligned}
		\right.
	\ee
\end{defi}

\begin{remark}\label{R:quattroo}
Consider a Lipschitz vector field $v$ and denote by $L$ its Lipschitz constant. Then it is well-known that:
\begin{enumerate}
\item The flow is invertible, and it holds $(\Phi^v_t)^{-1}=\Phi^v_{-t}$;
\item For all $T>0$ and $t\in [-T,T]$, the map $\Phi^v_t:\R^n\to\R^n$ is a Lipeomorphism with Lipschitz constant $e^{LT}$. This is a consequence of Gr\"onwall's inequality.
\item The map $\Psi(x,t):=\Phi^v_t(x):\R^n\times\R\to\R^n$ is locally Lipschitz continuous. As a consequence, for $N\subset \R^n\times\R$ the condition $\Ha^n(N)=0$ implies $\Ha^n(\Psi(N))=0$. 
\end{enumerate}
\end{remark}
Our proof of Theorem~\ref{teo_principale} only relies on the above properties of the flow.

We are now ready to define the evoluted set $S^t$, that is the main object that is studied in the article.

	\begin{defi}[Evoluted set] \label{d-ev}
	Let $v$ be a Lipschitz vector field on $\R^n$ and $\Phi^v_t$ its flow. Given a set $S\subset \R^n$ and $t\geq 0$, we define the {\em evoluted set}
	\be
		S^t:=\bigcup_{\tau\in [0,t]}\Phi_\tau^v(S).
	\ee
	\end{defi}
	
	\section{Counterexamples} \label{s-ex}
		
In this section, we provide counterexamples that show that the hypotheses of the main Theorem~\ref{teo_principale} are close to sharpness. In Example~\ref{ex-plane}, we will show that the condition $\Ll^n(\partial S)=0$ does not ensure $\Ll^n(\partial (S^t))=0$ even with $S$ open bounded and $v$ very regular (indeed constant).
In Example~\ref{ex-unbounded}, we provide a variation of Example~\ref{ex-plane}, focusing on the infinite time horizon: we define an unbounded set $A_u$, satisfying the hypothesis of Theorem~\ref{teo_principale}, whose evoluted set $S^{+\infty}_u$ has non-neglibigle boundary.

In Example~\ref{ex-infty}, very different from the previous ones, we prove again that the evoluted set for $t=+\infty$ can have non-neglibigle boundary even when the hypotheses of Theorem~\ref{teo_principale} are satisfied. We start with a bounded connected set $S$, but with a vector field which is more complex and vanishes outside $S^\infty$. An easy variation of this example would show that, for {\em every} connected and simply connected bounded open set $\Omega\subset\R^2$, there exist a Lipschitz vector field $v$ on $\R^2$ and a ball $S=B(p,r)\subset\R^2$ such that $S^\infty=\Omega$.

\begin{example}[$\Ll^n(\partial S)=0$ does not ensure $\Ll^n(\partial (S^t))=0$]\label{ex-plane}
Consider the sequence of positive natural numbers $i=1,2,\ldots$ and write each of them as $i=2^n+k$ with $n$ the maximal natural exponent $\lfloor\log_2(i)\rfloor$ and $k=0,\ldots,2^{n}-1$. Define the corresponding sequence $(x_i)_{i\geq 1}$ by $x_i=\frac{k}{2^n}$, where $n,k$ are given above. It is easy to observe that $x_i$ is a dense sequence in $[0,1]$.

We are now ready to define the set $S\subset \R^2$ in Figure~\ref{fig:controesempio_piano} as follows:
\begin{align}
&S:= \cup_{i=1,2,\ldots} B(c_i,r_i), 
&&c_i=(x_i,0),
&&r_i:=2^{-i-2}\ .
\end{align}
Notice that the projection of $S$ on the $x$-axis is the relatively open set
\begin{align} 
&O=\cup_{i\geq 1}(x_i-r_i,x_i+r_i), 
&&
\Ll^1(O)\leq \sum_{i=1}^{+\infty}2r_i=\frac12\ . 
\end{align} 
The boundary $\partial S$ is contained in the union of the circles $\partial B(c_i,r_i)$ and the segment $[0,1]\times \{0\}$. In particular, it is contained in a countable union of sets with zero $\Ll^2$-measure, then it satisfies $\Ll^2(\partial S)=0$. Moreover, $\partial S$ contains the compact Cantor-like set $K\times\{0\}$, where
\begin{align} 
K:= [0,1]\setminus O, 
&&\Ll^1(K)=1-\Ll^1(O)\geq\frac12\ .
\end{align} 
Indeed, each point in $[0,1]\times\{0\}$ is in the closure of $S$, since the sequence $x_i$ is dense in $[0,1]$, while $S$ is open and thus coincides with its interior.
We now choose the constant vector field $v=(0,1)$. 
As the projection of $S$ on the $x$-axis is $O$ then the boundary of the evoluted set $\partial (S^1)$ contains $K\times [0,1]$ and, in particular, $\Ll^2(\partial (S^1))\geq \Ll^1(K)\cdot 1>0$.

\begin{figure}
	\includegraphics[width=\linewidth]{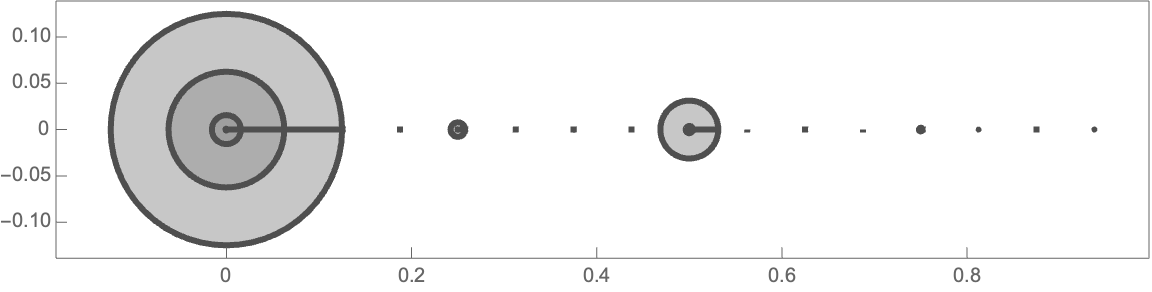}
\caption{The open set $S$ of Example~\ref{ex-plane}, which evolves with the constant vector field $v=(0,1)$.}
\label{fig:controesempio_piano}
\end{figure}
\end{example}

\begin{example}[Infinite time: initial unbounded set] \label{ex-unbounded}
With notations of Example~\ref{ex-plane}, define a new unbounded open set $A_u$ as follows:
\begin{align}
&A_u:= \cup_{i=1,2,\ldots} B(c_{i},r_i) 
&&\text{where now
$c_{i}=(x_i,-i)$.}
\end{align}
The projection $A_u$ on the $x$ axis is again $O$: in particular $\partial (A_u^{\infty})$ contains $ K\times[0,1]$, thus $\Ll^2(\partial (A_u^{\infty}))>0$. Remark that $A$ satisfies the hypotheses of Theorem~\ref{teo_principale}, thus $\Ll^2(\partial (A_u^{t}))=0$ for any finite time $t\in\R$.
\end{example}

\begin{example}[Infinite time: initial bounded set] \label{ex-infty} Consider an {\em Osgood curve}~\cite{Osgood} in the plane, i.e. a Jordan curve $\Gamma\subset\R^2\equiv \mathbb C$ such that $\mathscr L^2(\Gamma)>0$, and let $\Omega$ be the bounded connected open region enclosed by $\Gamma$. Let $D\subset\R^2\equiv \mathbb C$ be the unit (open) disk centered at the origin. Since $\Omega$ is also simply connected (\cite[page~2]{pommerenke}), by the Riemann Mapping Theorem~\cite[page~4]{pommerenke} there exists a biholomorphic map $G:D\to\Omega$. Moreover, by Carath\'eodory's theorem~\cite[Theorem~2.6]{pommerenke} $G$ can be extended to a homeomorphism $G:\overline D\to\overline\Omega$.

Let $g:[0,1)\to (0,+\infty)$ be a positive smooth function; $g$ will be chosen later in such a way that, as $r\to 1^-$, both $g(r)$ and $g'(r)$ decrease to 0 ``rapidly enough'', in a way depending on $G$ only. Consider the vector field $u(q):=g(|q|)q$ for $q\in D$; assuming also that $g$ is constant on $[0,1/2]$, then $u\in C^\infty(D)$. Moreover, the flow of $u$ is radial and it can be easily shown that, for every open neighborhood $U\subset D$ of $0$, it holds $U^\infty=\cup_{t\geq 0} \Phi^u_t(U)=D$.

Consider now the push-forward of $u$, i.e., the smooth vector field $v$ on $\Omega$ defined by $v(G(q)):=d G_q(u(q))$. We will later prove that $v$ is Lipschitz continuous on $\Omega$, hence it can be extended to a Lipschitz vector field in $\R^2$. The flow of $v$ on $\Omega$ is conjugated to the flow of $u$; in particular, choosing $S:=B(G(0),\delta)\subset\Omega$ and setting $U:=G^{-1}(S)$, we have that 
\[
S^\infty=\bigcup_{t\geq 0} \Phi^v_t(S) = G\left( \bigcup_{t\geq 0} \Phi^u_t(U)\right)=G(D)=\Omega.
\]
Therefore $\mathscr L^2(\partial(S^\infty))=\mathscr L^2(\partial\Omega)>0$ and we have only to prove that $v$ is Lipschitz continuous on $\Omega$. To this end, on differentiating the equality 
\[
v(p)=g(|G^{-1}(p)|)\; \nabla G(G^{-1}(p))\:G^{-1}(p)
\]
(where we use matrix notation) we obtain the estimate
\be\label{eq_stima}
\begin{split}
|\nabla v(p)|\leq
&\; |g'(|q|)|\,\,|\nabla G^{-1}(p)|\,\, |\nabla G(q)|\,\,|q|\\
&\; + g(|q|)\,\Big[\; |\nabla^2 G(q)|\,\,|\nabla G^{-1}(p)|\,\,|q|
+ |\nabla G(q)|\,\, |\nabla G^{-1}(p)|\; \Big],
\end{split}
\ee
where $q:=G^{-1}(p)$. 
Let us show that the quantity on the right hand side of~\eqref{eq_stima} can be bounded uniformly in $p\in\Omega$ with a proper choice of the function $g$; this will prove the Lipschitz continuity of $v$. 

For $r\in(0,1)$ define
\begin{align*}
m_0(r):=1+
\max_{q\in \overline{B(0,r)}}\Big\{&|\nabla G^{-1}(G(q))| \,\,|\nabla G(q)|\,\,|q|+ |\nabla^2 G(q)|\,\,|\nabla G^{-1}(G(q))|\,\,|q|\\
& + |\nabla G(q)|\,\, |\nabla G^{-1}(G(q))|
\Big\}
\end{align*}
and observe that~\eqref{eq_stima} gives
\be\label{eq_stima2}
|\nabla v(p)| \leq \big(|g'(|q_p|)|+g(|q_p|)\big)\:m_0(|q_p|)\qquad\forall\: p\in\Omega.
\ee
The function $m_0$ is continuous and non-decreasing on $(0,1)$; fix a smooth, non-decreasing function $m_1:(0,1)\to \R$ such that $m_1\geq m_0\geq 1$ and consider the smooth and decreasing function
\[
g_1(r):=\int_r^1\frac{1}{m_1(\rho)}\:d\rho,\qquad r\in(0,1).
\]
Eventually, let $g:[0,1)\to\R$ be a positive, non-increasing smooth function such that $g$ is constant on $[0,1/2]$ and $g=g_1$ on $[2/3,1)$; in particular, for every $r\in[2/3,1)$ we have 
\be\label{eq_stimesug}
\begin{split}
&|g'(r)|=\frac 1{m_1(r)}\leq \frac 1{m_0(r)}\\
&|g(r)|=\int_r^1\frac1{m_1(\rho)}d\rho \leq \frac{1-r}{m_1(r)}\leq \frac 1{m_0(r)}
\end{split}
\ee
due to the monotonicity of $m_1$. 
From~\eqref{eq_stima2} and~\eqref{eq_stimesug} we deduce
\begin{align*}
&|\nabla v(p)| \leq \left(\max_{[0,2/3]} \{|g'|+|g|\}\right)m_0(2/3)\quad
&&\text{if }p\in G(B(0,2/3)),\\
&|\nabla v(p)| \leq 2
&&\text{if }p\in \Omega\setminus G(B(0,2/3))
\end{align*}
and the Lipschitz continuity of $v$ follows.
\end{example}


\section{Proof of the main result}	 \label{s-main}

In this section, we prove our main theorem. We first prove two auxiliary results.	

\begin{prop}\label{prop_PUDCLipeo}
If $\Phi:\R^n\to\R^n$ is a Lipeomorphism and $S\subset\R^n $, then $\Phi(x)$ is a Lebesgue point of $ (\Phi(S)) ^\complement$ if and only if $x $ is a Lebesgue point of $S^\complement$.
\end{prop}
\begin{proof}
Fix $L>0$ and $M>0$ such that for every $y,z\in\R^n$
\[
\|\Phi(y)-\Phi(z)\|\leq L\|y-z\|\quad\text{and}\quad\|\Phi^{-1}(y)-\Phi^{-1}(z)\|\leq M\|y-z\|.
\]
Let us show that if $\Phi(x)$ is a Lebesgue point of $(\Phi(S))^\complement$ then $x$ is a Lebesgue point of $S^\complement$.
The other implication follows exchanging $\Phi$, $\Phi^{-1}$, $\Phi(S) $, $S$.
For $r>0$, it holds
\be
0\leq {\Ll^n(S\cap B(x,\tfrac{r }L))} \leq M^n\Ll^n(\Phi(S\cap B(x,\tfrac {r }L)))
\leq
 M^n\Ll^n(\Phi(S)\cap B(\Phi(x),r )).
\ee
As a consequence, if $\limsup_{r\to0^+} \Ll^n(\Phi(S)\cap B(\Phi(x),r ))/\Ll^n(B(\Phi(x),r ))=0$, then $\limsup_{r\to0^+} \Ll^n(S\cap B(x,r))/\Ll^n(B(x,r ))=0$. This proves the statement.
\end{proof}

\begin{lemma}\label{lem_bordi}
If $v$ is a Lipschitz vector field in $\R^n$ and $S\subset\R^n$ is any set, then
\be\label{eq_lembordi}
\text{for every $t>0$}\qquad\partial(S^t)\subset (\partial S)^t
\qquad\text{and }\qquad \overline{(\partial S)^t}= (\partial S)^t.
\ee
\end{lemma}
\begin{proof}
We first observe that if $C\subseteq \R^n$ is closed, then $ C^t$ is closed. Consider any sequence $p_n$ in $ C^t$ converging to some $\overline p$. 
By definition $p_n=\Phi^v_{\tau_n}(x_n)$ with $|\tau_n|\leq |t|$ and $x_n\in C$.
Recall that the inverse map $\Psi(p,-\tau)=(\Phi^v_{\tau})^{-1}(p)$ is continuous in both argument. Thus, eventually passing to a subsequence, let $(\tau_n)_n$ converge to some $\overline\tau$, then $x_n=\Psi(p_n,-\tau_n)$ converges to $\overline x:=\Psi(\overline p,-\overline\tau)$. As $x_n\in C$ and $C$ is closed, we conclude that $\overline p=\Phi^v_{\overline \tau}(\overline x)\in C^t$.

In particular, $(\partial S)^t$ and $(\overline S)^t$ are closed since $\partial S$ and $\overline S$ are closed. This already proves the second statement.

We now prove that $\partial(S^t)\subset (\partial S)^t$. 
Since $\Phi^v_\tau$ is an homeomorphism, then $(\INT(S))^t=\cup_{\tau\in[0,t]}\Phi^v_\tau(\INT(S))$ is open, being the union of open sets, so that $(\INT(S))^t\subset \INT(S^t)$.
Since $S^t\subset (\overline S)^t$ by definition of evoluted set and $(\overline S)^t$ is closed, then $\partial(S^t)\subset(\overline S)^t$. 
From $(\INT(S))^t\subset \INT(S^t)$ and $\partial(S^t)\subset(\overline S)^t$ we get
\[
\partial(S^t)=\partial(S^t)\setminus\INT(S^t)\subset \partial(S^t)\setminus(\INT(S))^t\subset (\overline S)^t\setminus(\INT(S))^t\subset(\partial S)^t\ .
\]
The last inclusion is just a consequence of the definition of evoluted set.
\end{proof}

\begin{remark}
When we consider the case $t=+\infty$, we lack precisely the property of Lemma~\ref{lem_bordi}, see Example~\ref{ex-infty}. Indeed, $(\partial S)^\infty$ is not anymore closed and $\partial (S^\infty)$, even though it is still subset of the closure of $(\partial S)^\infty$, might contain parts of the closure of $ (\partial S)^\infty$ of positive measure.
\end{remark}

%
%

We are now ready to prove our main result.

\begin{proof}[Proof of Theorem~\ref{teo_principale}]
The case $t=0$ is trivial. Moreover, since $\Phi^v_\tau=\Phi^{-v}_{-\tau}$, one can always assume $t>0$. Let us define $I:=\INT(S)$. By assumption we can write 
\be\label{eq_decomp}
\partial S= N\cup B,
\ee
where $\Ha^{n-1}(N)=0$ and 
\begin{align*}
B:=&\{p\in \partial S:p\text{ is not a Lebesgue point of }S^\complement\}\\
=&\{p\in \partial I:p\text{ is not a Lebesgue point of }I^\complement\}.
\end{align*}
The last equality follows from the fact that, since $\Ll^n(\partial S)=0$, the Lebesgue points of $S^\complement$ are exactly those of $I^\complement$.
Since $\Ha^n(N\times \R)=0$, Remark~\ref{R:quattroo} implies that
\be\label{eq_Ntnegl}
\Ll^n(N^t)=\Ll^n(\Psi(N\times[0,t]))=0.
\ee
We now prove that
\be\label{eq_Btnegl}
\Ll^n(B^t\setminus I^t)=0.
\ee
Fix $p\in B^t\setminus I^t$; by definition, there exist $x\in B\subset\partial I$ and $\tau\in[0,t]$ such that $p=\Phi^v_\tau(x)$. Since $x$ is not a Lebesgue point of $I^\complement$, then $p$ is not a Lebesgue point of the open set $(\Phi^v_\tau(I))^\C$. In particular
\begin{align*}
\liminf_{r\to 0^+}\frac{\Ll^n( (B^t\setminus I^t) \cap B(p,r))}{\Ll^n(B(p,r))}&\leq \liminf_{r\to 0^+}\frac{\Ll^n( B(p,r)\setminus I^t)}{\Ll^n(B(p,r))} \\
& \leq\liminf_{r\to 0^+}\frac{\Ll^n( B(p,r)\setminus \Phi^v_\tau(I))}{\Ll^n(B(p,r))} 
\ <\ 1.
\end{align*}
This proves that $B^t\setminus I^t$ has no Lebesgue points, thus~\eqref{eq_Btnegl} follows.

It is now easy to conclude the proof: since $I^t$ is the union of the open sets $\Phi^v_\tau(I)$, then $I^t$ is open and $I^t\subset \INT(S^t)$. Using Lemma~\ref{lem_bordi} and the decomposition~\eqref{eq_decomp}, the inclusions
\[
\partial(S^t)=\partial(S^t)\setminus \INT(S^t)
\subset 
\partial(S^t) \setminus I^t\subset (\partial S)^t \setminus I^t
\subset N^t \cup (B^t\setminus I^t),
\] 
together with~\eqref{eq_Ntnegl} and~\eqref{eq_Btnegl},
imply that $\Ll^n( \partial (S^t))=0$.
\end{proof}

Let us state some consequences of our main result.

\begin{cor} \label{c-c} Let $S$ be any set that satisfies one of the following conditions:
\begin{enumerate} 
\item $\Ha^{n-1}$-almost every point of $\partial S$ is not a Lebesgue point of $S^\complement$, as in Theorem~\ref{teo_principale};
\item $S$ is a Lipschitz domain (see e.g.~\cite[Definition 4.4]{EvansGariepy});
\item $S$ is a weakly Lipschitz domain (see e.g.~\cite[Definition 2.1]{AxMcIntosh});
\item $S$ is a set of finite perimeter for which $\Ha^{n-1}(\partial S\setminus\partial_* S)=0$, where $\partial_* S$ denotes the reduced boundary (see e.g.~\cite[Definition~3.54]{AFP}).
\end{enumerate}

Let $v$ satisfy one of the following conditions:
\begin{enumerate}
\item[{\rm (i)}] $v$ is autonomous and globally Lipschitz, as in Theorem~\ref{teo_principale};
\item[{\rm (ii)}] $v$ is autonomous locally Lipschitz continuous and with sub-linear growth (i.e., $\|v(x)\|\leq M(1+\|x\|)$ for every $x\in\R^n$), with $S$ moreover bounded;
\item[{\rm (iii)}] $v$ is a uniformly Lipschitz continuous and bounded {\em time-dependent} vector field $v=v(x,t):\R^n\times\R\to\R^n$.
\end{enumerate}

Then, $\Ll^n(\partial(S^t))=0$ for every $t\in\R$.
\end{cor}
\begin{proof} It is easy to prove that each of the hypotheses on $S$ implies that $\Ha^{n-1}$-almost every point of its boundary $\partial S$ is not a Lebesgue point of $S^\complement$. For the conditions 2--3 this is a direct consequence of the definition, by recalling that a Lipeomorphism maps an open set $S$ onto an open set $\Phi(S)$, and the boundary $\partial S$ into the boundary $\partial \Phi(S)$. For condition 4, it is enough to recall that $\lim_{r\to 0^+} \Ll^n(S\cap B(x,r))/\Ll^n(B(x,r))=1/2$ for every $x\in\partial_* S$. See e.g.~\cite[Theorem~3.61]{AFP}.

As for the hypotheses on $v$, the non-autonomous case is similar to the autonomous one, since the flow satisfies the same regularity hypotheses as well as Gr\"onwall's inequality. For case (ii), it is sufficient to observe that boundedness of $S$ and sublinearity of $v$ imply boundedness of $S^t$. Thus, one can change $v$ outside $\overline{S^t}$ and recover the same result.
\end{proof}

\begin{remark}
It is worth noticing that in the counterexamples to the negligibility of $\partial S^t$ provided in Examples~\ref{ex-plane}-\ref{ex-unbounded} and in~\cite{BoarottoRossi}, the open set $S$ has finite perimeter but $\Ha^{n-1}(\partial S\setminus \partial_*S)>0$. In fact, in these examples (many) points of $\partial S\setminus \partial_*S$ are Lebesgue points of $S^\complement$.
\end{remark}
\begin{remark} The corollary above also includes our previous results related to this problem. We proved in \cite{BoarottoRossi} that the statement holds when $S$ is bounded $C^{1,1}$, while we proved in \cite[Lemma 1.5]{mintime} that it holds when $S$ is open bounded and satisfying the uniform interior cone condition with $v$ of class $C^1$.
\end{remark}

	\begin{cor} Consider the evoluted set on the {\em open} time interval $B^t:=\cup_{\tau\in (0,t)}\Phi_\tau^v(S)$. The statements of Theorem~\ref{teo_principale} and Corollary~\ref{c-c} hold in this case too.
	\end{cor}
	\begin{proof}
	It is sufficient to observe that
\be\label{e-b}
\partial (B^t)\subset \partial(S^t)\cup\partial S\cup \partial(\Phi^v_t(S))
\ee
implies $$\Ll^n(\partial (B^t))\leq \Ll^n(\partial (S^t))+\Ll^n(\partial S)+\Ll^n(\partial (\Phi^v_t(S))=0,$$ where we used our main result, as well as Proposition~\ref{prop_PUDCLipeo} and the fact that $\Ll^n(\partial S)=0$. We also need to prove~\eqref{e-b}. First observe that a point $b\in\partial (B^t)$ is both the limit of a sequence $b_n=\Phi_{\tau_n}^v(x_n)$ with $(x_n,\tau_n)\in S\times (0,t)$ and of a sequence $y_n\not\in S^t$.
	Since $\tau_n$ is bounded, we can consider a subsequence (that we do not relabel) so that $\tau_n\to \tau\in [0,t]$. We have three possibilities:
\begin{itemize}
	\item
	If $\tau\in(0,\tau)$ we already have the thesis, since $b$ is the limit of both $y_n\not\in S^t$ and $b_n=\Phi_{\tau_n}^v(x_n)$ which definitively belongs to $S^t$.
	\item 
	If $\tau=0$, then $x_n=\Psi(b_n,-\tau_n)$ and $b_n=\Psi(b_n,0)$ have the same limit $b=\Psi(b,0)$ by continuity of $\Psi$ in $B(b,1)\times[-1,1]$. Then $b$ is the limit of $x_n\in S$ and $y_n\not\in S^t\supset S$, then $b\in\partial S$.	
	\item The case $\tau=t$ is similar to the previous case $\tau=0$: one can prove that $b\in \partial\Phi^v_t(S)$.
	\end{itemize}
	\end{proof}

	\bibliographystyle{abbrv}
	\bibliography{Biblio}
	
\end{document}